\documentclass{amsart}
\usepackage{amsfonts,amssymb,amscd,amsmath,enumerate,verbatim,calc}
\usepackage{xy}

\newcommand{\wrt}{with respect to}

\newcommand{\E}{\mathcal{E} }
\newcommand{\T}{\mathcal{T} }

\newcommand{\m}{\mathfrak{m} }

\newcommand{\K}{\mathbb{K} }

\newcommand{\bP}{\mathbb{\partial}}
\newcommand{\rt}{\rightarrow}

\newcommand{\X}{\mathbf{X} }
\newcommand{\Y}{\mathbf{Y} }

\newcommand{\ZZ}{\mathbb{Z} }

\newcommand{\Tor}{\operatorname{Tor}}

\newcommand{\Hom}{\operatorname{Hom}}
\newcommand{\Ext}{\operatorname{Ext}}

\theoremstyle{plain}

\newtheorem{theorem}{Theorem}[section]
\newtheorem{corollary}[theorem]{Corollary}
\newtheorem{lemma}[theorem]{Lemma}
\newtheorem{proposition}[theorem]{Proposition}

\newtheorem{conjecture}[theorem]{Conjecture}

\theoremstyle{definition}
\newtheorem{definition}[theorem]{Definition}
\newtheorem{s}[theorem]{}
\newtheorem{remark}[theorem]{Remark}
\newtheorem{example}[theorem]{Example}
\newtheorem{property}[theorem]{Property}

\theoremstyle{remark}

\begin{document}

\title{On derived functors  of Graded local cohomology modules}

\author{Tony J. Puthenpurakal}

\address{Department of Mathematics, Indian Institute of Technology Bombay, Mumbai 400076, India}

\email{tputhen@math.iitb.ac.in}

\author{Jyoti Singh}
\address{Department of Mathematics, Indian Institute of Technology Bombay, Mumbai 400076, India}
\email{jyotijagrati@gmail.com}

\date{\today}


\begin{abstract}
Let $K$ be a field of characteristic zero and 
let $R=K[X_1, \ldots,X_n ]$, with  standard grading. 
Let $\mathfrak{m}= (X_1, \ldots, X_n)$ and let $E$ be the $^*$injective hull of $R/\mathfrak{m}.$ 
Let $A_n(K)$ be the $n^{th}$ Weyl algebra over $K$. 
Let $I, J$ be homogeneous ideals in $R$. Fix $i,j \geq 0$ and set $M = H^i_I(R)$ and $N = H^j_J(R)$ considered as left 
$A_n(K)$-modules.
 We show the following two results for which no analogous result is known in characteristic $p > 0$.
\begin{enumerate}
 \item 
 $H^l_\mathfrak{m}(\Tor^R_\nu(M, N)) \cong E(n)^{a_{l,\nu}}$ for some $a_{l,\nu} \geq 0$.
 \item
For all $\nu \geq 0$; the finite dimensional vector space $\Tor^{A_n(K)}_\nu( M^\sharp, N)$ is concentrated in degree $-n$
(here $M^\sharp$ is the standard right $A_n(K)$-module associated to $M$).
\end{enumerate}
We also conjecture that for all $i \geq 0$  the finite dimensional vector space
$\Ext^i_{A_n(K)}(M, N)$ is concentrated in degree zero. We give a few examples which support 
this conjecture.
\end{abstract}
\maketitle

\section{Introduction}
\s \label{setup}\textit{Setup:} Let $K$ be a field of characteristic zero and let $R = K[X_1,\ldots,X_n]$.  
We give the standard grading on $R$. Let $\mathfrak{m}= (X_1, \ldots, X_n)$ and let $E$ be the $^*$injective hull of $R/\mathfrak{m}.$  Let
$A_n(K)  = K<X_1,\ldots,X_n, \partial_1, \ldots, \partial_n>$  be the $n^{th}$ Weyl algebra over $K$. 
We can consider $A_n(K)$ graded by giving $ \deg X_i = 1$  and $\deg \partial_i = -1$ for 
 $i = 1,\ldots, n$. Let $I, J$ be  homogeneous ideals in $R$. By a result due
to Lyubeznik, see \cite[2.9]{Lyubeznik},
 for each $i \geq 0$ the local cohomology module $H^i_I(R)$ is a 
 graded left holonomic $A_n(K)$-module.  Fix $i, j \geq 0$. Set $M = H^i_I(R)$ and $N = H^j_J(R)$.
Set $M^\sharp$ to be the standard right $A_n(K)$-module associated to $M$, see \ref{left-right}.

A few years ago the first author proved that the De Rham cohomology  \\
$H^\nu(\partial, H^i_I(R))$ is concentrated in degree $-n$ for each $\nu \geq 0$.
For all $\nu \geq 0$ there is a graded isomorphism
\begin{equation*}
H^\nu\left(\partial, H^i_I(R) \right)  \cong \Tor^{A_n(K)}_{n-\nu}(R^r, H^i_I(R)).  \tag{ * }
\end{equation*}
where $R^r$ is $R$ considered as a right $A_n(K)$-module by the isomorphism
$R^r = A_n(K)/(\partial)A_n(K)$. Consider $ ^lR$ to be $R$ considered as a left 
$A_n(K)$-module  via the isomorphism $ ^lR \cong A_n(K)/A_n(K)(\partial)$. It is elementary to note that $( ^lR)^\sharp \cong R^r$. We also note that $H^0_{(0)}(R) =\ ^lR$. Our first result is a considerable generalization of $(*)$.
\begin{theorem}\label{first}[with hypotheses as in \ref{setup}]
Fix $\nu \geq 0$. Then the graded $K$-vector space  $\Tor^{A_n(K)}_\nu(M^\sharp, N)$ is concentrated in degree $-n$, i.e., $\Tor^{A_n(K)}_\nu(M^\sharp, N)_j = 0$ for all $j \neq -n$.
\end{theorem}
\begin{remark}
We should note that $M^\sharp$ is a holonomic right $A_n(K)$-module. So by 
\cite[1.6.6]{Bjork},
$\Tor^{A_n(K)}_\nu(M^\sharp, N)$ is a finite dimensional $K$-vector space.
\end{remark}

By a remarkable result of Ma and Zhang \cite[1.2]{MaZhang}, if $M = H^i_I(R)$ is supported at $\m$ then  $M \cong E(n)^l$ for some $l \geq 0$ (this result also holds in characteristic $p > 0$). Our next result is a considerable generalization of this fact.
\begin{theorem}\label{second}[with hypotheses as in \ref{setup}.] Fix $\nu \geq 0$.
If $\Tor^R_\nu(M,N)$ is supported at $\m$ then $\Tor^R_\nu(M,N) \cong E(n)^l$ for some $l \geq 0 $. In particular if $\sqrt{I+J} = \m$ then for all $\nu \geq 0$ there exists $l_\nu \geq 0$ such that $\Tor^R_\nu(M,N) \cong E(n)^{l_\nu}$.
\end{theorem}

In view of Theorems \ref{first} and \ref{second} a natural question is to investigate \\  $\Ext^\nu_{A_n(K)}(H^i_I(R), H^j_J(R))$ for $\nu \geq 0$.  We note that by Bjork  \cite[2.7.15  \& 1.6.6]{Bjork} this is a finite dimensional $K$-vector space. Although we are unable to resolve this question, we make the following:
\begin{conjecture}\label{conj}[with hypotheses as in \ref{setup}.] Fix $\nu \geq 0$. The graded $K$-vector space $\Ext^\nu_{A_n(K)}(M, N)$ is concentrated in degree zero, i.e., $\Ext^\nu_{A_n(K)}(M, N)_j = 0$  for $j \neq 0$.
\end{conjecture}
We give evidence which support this conjecture.
\s\textit{Techniques used to prove our results} \\
 Ma and Zhang, \cite{MaZhang},  introduced the notion of Eulerian $D$-modules both in characteristic zero and in characteristic $p$ (here $D$ is the ring of differential operators on $K[X_1,\ldots, X_n])$. They showed that the local cohomology module $H^i_I(R)$ is Eulerian for all $i \geq 0$. Unfortunately however the class of Eulerian $D$-modules is not closed under extensions (see 3.5(1) in \cite{MaZhang}). 
To rectify this  the first author introduced the notion generalized Eulerian $D$-modules (in characteristic zero). We prove Theorem's \ref{first} and \ref{second}
in the general context of graded Lyubeznik functor's on $^*Mod(R)$ (see Section 2 for definition of graded Lyubeznik functor). 

In characteristic $p>0$, Ma and Zhang proved a considerably stronger result. 
They showed that every graded $F$-module is Eulerian (see Theorem 4.4 in \cite{MaZhang}). So if $\mathcal{T}$ is any graded Lyubeznik functor on $^*Mod(R)$ then $\mathcal{T}(R)$ is Eulerian. As a consequence, we have $H_\mathfrak{m}^i\mathcal{T}(R)\cong E(n)^{a_i}$ for some $a_i\geq 0$. A natural question is what happens when $\text{char} (K) = 0$. Note that in $\text{char}~ p > 0$, Ma and Zhang \cite{MaZhang} essentially use the Frobenius which is unavailable in characteristic zero.  It is perhaps hopeless to determine whether $\mathcal{T}(R)$ is Eulerian in characteristic zero. An essential obstruction is that 
Eulerian $D$-module are not closed under extensions. 

Our third result is
\begin{theorem}\label{third}[with hypotheses as in \ref{setup}.]
 Let $\mathcal{T}$ be a graded Lyubeznik functor on $^*Mod(R)$. Then  $\mathcal{T}(R)$ is a generalized Eulerian $A_n(K)$-module.  
\end{theorem}

Let  $H^\nu(\partial_r, \partial_{r+1}, \ldots, \partial_n, V)$ 
($H^\nu(X_r, \ldots, X_n, V)$) be the $\nu^{th}$ partial De Rham (and Koszul) cohomology of an $A_n(K)$-module $V$. Notice they are graded $A_n(K)$-modules if $V$ is a graded $A_n(K)$-module. The essential technique to prove Theorems \ref{first} and \ref{second} is the following result:
\begin{theorem}\label{induct}
 Let $V$ be a generalized Eulerian $A_n(K)$-module.
 Then for $r \geq 1$
 \begin{enumerate}[\rm (1)]
  \item 
  $H^\nu(\partial_r, \partial_{r+1}, \ldots, \partial_n ; V)(-n+ r - 1)$ is a generalized Eulerian module $A_{r-1}(K)$ for all $\nu \geq 0$.
  \item
  $H^\nu(x_r, \ldots, x_n ; V)$ is a generalized Eulerian $A_{r-1}(K)$ module for all $\nu \geq 0$.
  \item
  $H^\nu(x_1, \ldots, x_{r-1}, \partial_r, \partial_{r+1}, \ldots, \partial_n ; V)$ is a $K$-vector space concentrated in degree
  $-n + r - 1$
  \end{enumerate}
\end{theorem}

Here is an overview of the contents of the paper. In section two we discuss a few preliminaries that we need. In the next section
we discuss some properties of generalized Eulerian modules. In particular we show that (homogeneous) localization of generalized 
Eulerian modules are generalized Eulerian. In section four we prove Theorem \ref{third} and give some easy corollaries of this 
result. In section five we investigate Koszul homology of genaralized Eulerian modules. An easy consequence of our result
is that if a generalized Eulerian $A_n(K)$-module $M$ is finitely generated as an $R$-module then it is either zero or 
$R^m$ for some $m \geq1 $. We also indicate a proof of Theorem \ref{induct}. In section six we prove Theorem \ref{second}.
In the next section we prove Theorem \ref{first}. Finally in section eight we give evidence to support Conjecture \ref{conj}.

\section{Preliminaries}
In this section, we discuss a few preliminary result that we need.
Let $K$ be a field (not necessarily of characteristic zero) and let $R=K[X_1, \ldots, X_n]$ with standard grading.
Let $\mathfrak{m}=(X_1, \ldots, X_n)$ and let $E$ be the injective hull of $R/\mathfrak{m}$. Let $ \ ^* Mod(R)$ denote the category 
of graded $R$-modules.

\s  \textbf{Graded Lyubeznik functors:} \\
In this subsection, we define graded Lyubeznik functors.  We say $Y$ is \textit{homogeneous }closed subset of $\text{Spec}(R)$ if $Y= V(f_1, \ldots, f_s)$, where $f_i's$ are homogeneous polynomials in $R$.

We say $Y$ is a homogeneous locally closed subset of $\text{Spec}(R)$ if $Y = Y''-Y'$, where $Y', Y''$ are
homogeneous closed subset of $\text{Spec}(R)$. We have an exact sequence of  functor's on $\ ^*Mod(R)$:
\begin{equation}\label{eq1} H_{Y'}^i(-) \longrightarrow H_{Y''}^i(-) \longrightarrow H_{Y}^i(-) \longrightarrow
H_{Y'}^{i+1}(-).
\end{equation}

\begin{definition}
\textit{A graded Lyubeznik functor} $\mathcal{T}$ is a composite functor of the form $\mathcal{T}= \mathcal{T}_1\circ\mathcal{T}_2 \circ \ldots\circ\mathcal{T}_k$, where each $\mathcal{T}_j$ is either $H_{Y_j}^i(R)$, where $Y_j$ is a homogeneous locally closed subset of $\text{Spec}(R)$ or the kernel of any arrow appearing in (\ref{eq1}) with $Y'=Y_j'$ and $Y''= Y_j''$, where $Y_j' \subset Y_j''$ are two homogeneous  closed subsets of $\text{Spec}(R)$.
\end{definition}

 \textit{We now assume that characteristic of $K$ is zero.}
 
  Let $A_n(K)$ be the $n^{th}$ Weyl algebra over $K$. Set $\deg \partial_i= -1$. So $A_n(K)$ is a graded ring with $R$ as a graded subring.
 \begin{remark}
Although many results are stated for de Rham cohomology of a
$A_n(K)$-module $M$, we will actually work with de Rham homology.
 Let $S = K[\partial_1,\ldots,\partial_n]$. Consider it as a subring of $A_n(K)$.
 Then note that $H_i(\bP; M)$ is the $i^{th}$ Koszul homology module of $M$ with respect to $\bP$. 
 More generally if $u_1,\ldots, u_l$ are commuting elements in $A_n(K)$ then we can consider the subring 
 $K[u_1,\ldots, u_l]$ of $A_n(K)$ and consider the Koszul homology 
 $H_i(u_1, \ldots, u_l; M)$ for all $i \geq 0$.
\end{remark}

\s \textbf{ Koszul homology:} 
Let $M$ be a graded $A_n(K)$-module (not necessarily finitely generated). Let $u_1,\ldots, u_c$ be homogeneous commuting 
elements in $A_n(K)$.
Let $H_i(u_1,\ldots, u_c; M)$ be the $i^{th}$ Koszul homology module of $M$  with respect to $u_1,\ldots, u_c $ (with its natural grading). The following result is well-known.
\begin{lemma}\label{exact}
Let $ \mathbf{u} = u_r,  u_{r+1}, \ldots, u_c$ and let $\mathbf{u^\prime}  = u_{r+1}, \ldots, u_c$. 
Let $M$ be a graded left $A_n(K)$-module. Consider the (commutative) subring 
 $S = K[u_1,\ldots, u_l]$ of $A_n(K)$.  For each $i\geq 0$ there exists an exact sequence of graded $S$-modules.

$$ 0 \rightarrow H_0(u_r; H_i(\mathbf{u^\prime}; M))\rightarrow H_i( \mathbf{u}; M) 
\rightarrow H_1(u_r; H_{i-1}(\mathbf{u^\prime} ; M)) \rightarrow 0.$$
\end{lemma}

\s \label{left-right} \textit{ Standard method for converting left $A_n(K)$-modules into right modules and vice-versa:}
Consider the anti-isomorphism $\tau \colon A_n(K) \rt A_n(K)$ where $\tau(h\partial^\alpha) = (-1)^{|\alpha|}\partial^\alpha h$
(here $h \in R)$. Notice $\tau^2 = 1$ and $\tau(uv) = \tau(v)\tau(u)$ for all $u, v \in A_n(K)$. 

Let $M$ be a left $A_n(K)$-module, Consider $M^\sharp $ which is $M$ as an abelian group. However $M^\sharp$ is given a right $A_n(K)$-module structure
as follows:
\[
 m\circ u = \tau(u)m \quad \text{for} \ m \in M \ \text{and} \ u \in A_n(K).
\]
Similarly if $N$ is  right $A_n(K)$-module then using $\tau$ we can naturally define a left $A_n(K)$-module $^\sharp N$.  If $M$ is a left $A_n(K)$-module then
it is trivial to check that $^\sharp(M^\sharp) \cong M$.

Finally if $M = \bigoplus_{i \in \ZZ }M_i$ is a graded left $A_n(K)$-module then
$M^\sharp$ is also graded with $M^\sharp = \bigoplus_{i \in \ZZ }M_i$.
 
\textit{Example:} Let $ ^lR = A_n(K)/A_n(K)\partial$ be $R$ considered as a left $A_n(K)$-module. Also let $R^r = A_n(K)/\partial A_n(K)$
be $R$ considered as a right $A_n(K)$-module. It is not difficult to prove that $^lR^\sharp \cong R^r$ as right $A_n(K)$-modules.

\section{generalized Eulerian $A_n(K)$-modules}

In this section, we recall the notion of generalized Eulerian $A_n(K)$-modules from \cite{Puthenpurakal2}. The main result of this section is that  localization of graded generalized Eulerian $A_n(K)$-modules (with respect to a multiplicatively closed set of homogeneous elements) are generalized Eulerian.

Let $K$ be a field of characteristic zero and  let $R=K[X_1, \ldots, X_n]$  with standard grading. Let $A_n(K)$ be the $n^{th}$ Weyl algebra over $K$. Set $\deg \partial_i= -1$. So $A_n(K)$ is a graded ring with $R$ as a graded subring.
The Euler operator, denoted by $\mathcal{E}_n$, is defined as 
$$\E_{X,n} = \mathcal{E}_n  := \sum_{i=1}^n X_i\partial_i. $$

Note that $\deg \mathcal{E}_n=0$. Let $M$ be a graded $A_n(K)$-module. If $m\in M$ is homogeneous element, set $|m|= \deg m$.
\begin{definition}
 Let $M$ be a graded $A_n(K)$-module. We say $M$ is an \textit{Eulerian} $A_n(K)$-module if for any homogeneous $z$ in $M$
 we have
$$ \mathcal{E}_nz= |z|\cdot z.$$
\end{definition}
Clearly $R$ is an Eulerain $A_n(K)$-module.
\begin{definition}
A graded $A_n(K)$-module $M$ is said to be \textit{generalized Eulerian} if for any homogeneous element $z$ of $M$
there exists a positive integer $a$ (depending on $z$) such that
$$ (\mathcal{E}_n- {|z|})^a \cdot z = 0.$$
\end{definition}
The following properties of generalized Eulerian modules were proved in \cite{Puthenpurakal2}. 

\begin{property}[Proposition 2.1 in \cite{Puthenpurakal2}] \label{property1} Let $0 \rightarrow M_1 \overset{{\alpha_1}}{\rightarrow} M_2 \overset{{\alpha_2}}{\rightarrow} M_3 {\rightarrow} 0 $ be a short exact sequence of graded $A_n(K)$-modules. Then $M_2$ is generalized Eulerian
 if and only if $M_1$ and $M_3$ are generalized Eulerian.

\end{property}

If $M$ is graded $A_n(K)$-module, then for $l\in \mathbb{Z}$ the modules $M(l)$ denotes the shift of $M$ by $l$; that is, $M(l)_n=M_{n+l}$ for all $n\in \mathbb{Z}.$
\begin{property}[Proposition 2.2 in \cite{Puthenpurakal2}]\label{property2}
Let $M$ be a non-zero generalized Eulerian $A_n(K)$-module. Then the shifted module $M(l)$ is not a generalized Eulerian $A_n(K)$-module for $l \neq 0$.
\end{property}
The goal of this section is to prove that localization of a generalized Eulerian $A_n(K)$-module is generalized Eulerian. We first prove: 

\begin{lemma}\label{property3}
Let $M$ be a graded $A_n(K)$-module. Consider a homogeneous polynomial $f \in R$ and let $w$ be a homogeneous element of $M$. Then
$$ (\mathcal{E}_n- |w| + |f|)\frac{w}{f} = \frac{1}{f}(\mathcal{E}_n- |w|)w.$$
\end{lemma}
\begin{proof} First we compute
$$\begin{array}{lll}
\mathcal{E}_n\dfrac{w}{f} &=& \sum x_i\partial_i(\dfrac{w}{f}),\\
&=& \sum x_i \bigg{(}\dfrac{f\partial_i(w)- \partial_i(f)w }{f^2} \bigg{)},\\
&=& \frac{f}{f^2}(\sum x_i\partial_i(w))- \dfrac{1}{f^2}(\sum x_i\partial_i(f))w,\\
&=& \dfrac{1}{f}\mathcal{E}_nw - \dfrac{1}{f^2}(|f|\cdot f) w,\\
&=& \dfrac{1}{f}(\mathcal{E}_n-|f|)w.
\end{array}$$
Now we have
$$\begin{array}{lll}
(\mathcal{E}_n- |w| + |f|)\frac{w}{f}&=&  \dfrac{1}{f}(\mathcal{E}_n-|f|)w - (|w| - |f|)\dfrac{w}{f},\\
&=& \frac{1}{f}\mathcal{E}_nw- \frac{|w|}{f}w,\\
&=& \frac{1}{f}(\mathcal{E}_n- |w|)w.
\end{array}$$
This finishes the proof.
\end{proof}

We now extend the previous result:

\begin{lemma}
Let $M$ be a graded $A_n(K)$-module. Consider a homogeneous polynomial $f \in R$ and let $z$ be a homogeneous element of $M$. Then
$$ (\mathcal{E}_n- |z| + |f|)^n\frac{z}{f} = \frac{1}{f}(\mathcal{E}_n- |z|)^nz, ~~~~~ \text{for all}~~~~~ n\geq 1. $$

\end{lemma}

\begin{proof}
We prove this result by induction on $n$. By Lemma \ref{property3}, the result is true for $n=1$. Now we assume for $n=r$, so we have
$$ (\mathcal{E}_n- |z| + |f|)^r\frac{z}{f} = \frac{1}{f}(\mathcal{E}_n- |z|)^rz,$$
{for all} $r\geq 1$. Put $w= (\mathcal{E}_n-|z|)^rz$. We observe that $\mathcal{E}_n- |z|$ is homogeneous of degree $0$. This implies that $$\text{deg}(\mathcal{E}_n- |z|)^r=0. $$ Hence $$|w| = |z|.$$ 
Now we have 
$$\begin{array}{lll}
(\mathcal{E}_n- |z| + |f|)^{r+1}\frac{z}{f} &=& (\mathcal{E}_n- |z| + |f|)(\mathcal{E}_n- |z| + |f |)^{r}\frac{z}{f},\\
&=& (\mathcal{E}_n- |z|+ |f|)\frac{w}{f},\\
&=& (\mathcal{E}_n-|w| + |f|)\dfrac{w}{f},\\
&=& \frac{1}{f}(\mathcal{E}_n- |w|)w \hspace{2cm}(\text{{by}} \text{ Lemma}~~ \ref{property3}),\\
&=& \frac{1}{f}(\mathcal{E}_n- |z|) (\mathcal{E}_n- |z |)^rz,\\
&=& \frac{1}{f}(\mathcal{E}_n- | z |)^{r+1}z. 
\end{array}$$
So the result holds by induction.

\end{proof}

\begin{corollary}\label{property4}
Let $M$ be a  generalized Eulerian $A_n(K)$-module. Then $S^{-1}M$ is also a generalized Eulerian $A_n(K)$-module for each homogeneous system $S\subseteq R$. In particular, $M_f$ is generalized Eulerian for each homogeneous polynomial $f \in R$.
\end{corollary}
\begin{proof}
Let $z\in M$ be homogeneous and let $f\in S$. Put $\xi= \dfrac{z}{f}$, so $|\xi| = |z| - |f|$. Since $M$ is generalized Eulerian, so 
$(\mathcal{E}_n- |z|)^az=0~~~~~~~~ \text{for some}~~~ a\geq 1. $ Now we have
$$\begin{array}{lll}
(\mathcal{E}_n- |z| + |f|)^a\dfrac{z}{f} &=&\dfrac{1}{f}(\mathcal{E}_n- |z| )^az,\\
&=& \dfrac{1}{f}\cdot0,\\
&=&0.
\end{array}$$
This implies that $$(\mathcal{E}_n- | \xi |)^a\xi=0.$$ 
Thus $S^{-1}M$ is generalized Eulerian $A_n(K)$-module.
\end{proof}

\begin{corollary}\label{property5}
Let $I= (f_{1}, \ldots, f_{s})$ be a homogeneous ideal in $R$ with $f_i'$s are  homogeneous polynomials in $R$. Let $M$ be a generalized Eulerian $A_n(K)$-module. Then $H_I^i(M)$ is a generalized Eulerian $A_n(K)$-module for all $i \geq 0$.
\end{corollary}

\begin{proof}
Let $C.$ be the {$\check{C}e$ch} complex of $M$ with respect to $f_1, \ldots, f_s$. 
This is a complex of graded $A_n(K)$-modules. By Corollary \ref{property4}, $M_{f_{1}, \ldots, f_{s}}$ 
is generalized Eulerian. So each module $C^j$ in $C.$ is generalized Eulerian. 
So by Property \ref{property1}, $H_I^i(M)$ is  a generalized Eulerian $A_n(K)$-module for all $i \geq 0$.

\end{proof}
Our final result is that the Eulerian operator is stable under a linear change of variables.

\s \label{change of variables} For $i = 1,\ldots,n$ let $Y_i = \sum_{ j = 1}^{n}b_{ij}X_j$ be a linear change of variables. Set $B = (b_{ij})$. Furthermore
set $\X = (X_1,\ldots, X_n)^{tr}$ and $\Y = (Y_1, \ldots, Y_n)^{tr}$. Set 
\[
\partial_X = (\partial/\partial X_1, \cdots, \partial/\partial X_n)^{tr} \quad \text{and} \quad \partial_Y = (\partial/\partial Y_1, \cdots, \partial/\partial Y_n)^{tr}. 
\]
Let $\E_X =  \X^{tr}\partial_X$ be the Eulerian operator \wrt \ $\X$ and let
 $\E_Y$ be the Eulerian operator \wrt \ $\Y$.  We have
\begin{proposition}\label{change-equal}
(with hypotheses as in \ref{change of variables})
$$\E_X = \E_Y.$$
\end{proposition}
\begin{proof}
It can be easily verified that $\partial_X = B^{tr}\partial_Y$.
Now notice
\begin{align*}
\E_X &= \X^{tr} \partial_X \\
&= \Y^{tr} (B^{-1})^{tr}B^{tr}\partial_Y \\
 &= \Y^{tr} \partial_Y \\
 &= \E_Y.
\end{align*}
\end{proof}

\section{proof of Theorem \ref{third}}
In this section we give a proof of Theorem \ref{third}. We also give a few applications.

\begin{proof}[Proof of Theorem \ref{third}]
It suffices to show that if $\mathcal{T}$ is a graded Lyubeznik functor
and $M$ is graded generalized Eulerian $A_n(K)$-modules, then so is $\mathcal{T}(M)$. 
Since $\mathcal{T}= \mathcal{T}_1\circ \mathcal{T}_2 \circ \ldots \circ \mathcal{T}_s$, 
by induction it is enough to show that $\mathcal{T}_i(M)$ is a generalized Eulerian $A_n(K)$-module.
But $\mathcal{T}_i(M)$ is a graded $A_n(K)$-submodule of $H_Y^i(M)$, where $Y$ is locally closed homogeneous
closed subset of $\text{Spec}(R)$. 
By Property \ref{property1}, it  suffices to show that $H_Y^i(M)$ is generalized Eulerian. 
Let $Y= Y''-Y'$, where $Y', Y''$ are homogeneous closed sets of $\text{Spec}(R)$. 
Then we have an exact sequence of graded $A_n(K)$-modules
\begin{equation} 
H_{Y'}^i(M) \longrightarrow H_{Y}^i(M) \longrightarrow H_{Y''}^{i+1}(M). 
\end{equation}
By Corollary \ref{property5}, $H_{Y'}^i(M)$ and $H_{Y''}^{i+1}(M)$ are generalized Eulerian $A_n(K)$-modules. Thus by Property \ref{property1}, $H_Y^i(M)$ is generalized Eulerian $A_n(K)$-module.
 
\end{proof}

\begin{remark}
Let $E$ be $^*$injective hull of $R/\mathfrak{m}$.  By Theorem 1.1(2) in \cite{MaZhang}, $E(n)$ is an Eulerian $A_n(K)$-module. 
\end{remark}

\begin{corollary}
Let $K$ be a field of characteristic zero. Let $R=K[X_1,\ldots, X_n]$ be standard graded. Let $\mathcal{T}$ be a graded Lyubeznik functor on $^*Mod(R)$. Then  $H_\mathfrak{m}^i\mathcal{T}(R)= E(n)^{a_i}$ for some $a_i\geq 0.$ 
\end{corollary}

\begin{proof}
By Proposition 5.5 in \cite{MaZhang},
$$ H_\mathfrak{m}^i(\mathcal{T}(R))\cong\oplus_j ^*{E}(n_j).$$
By Lyubeznik's result (Corollary 2.14 in \cite{Lyubeznik}), the number of copies of $E(n_j)$ will be finite. Note $H_\mathfrak{m}^i\circ \mathcal{T}$ is a Lyubeznik functor. By Theorem  \ref{third},    $\oplus_j {E}(n_j)$ is a generalized Eulerian $A_n(K)$-module. Thus $E(n_j)$ is generalized Eulerian. So by Property \ref{property2}, $n_j=n$. Hence  we have $H_m^i\mathcal{T}(R)= E(n)^{a_i}$ for some $a_i\geq 0$.

\end{proof}
An easy consequence of Theorem 3.1 in \cite{Puthenpurakal2} is
\begin{corollary}\label{corollary}
Let $K$ be a field of characteristic zero. Let $R=K[X_1,\ldots, X_n]$ be standard graded.
Let $\mathcal{T}$ be a graded Lyubeznik functor on $^*Mod(R)$. 
Then de Rham cohomology  module $H^j(\partial, \mathcal{T}(R))$ is concentrated in degree $-n$, i.e.,
$$H^j(\partial, \mathcal{T}(R))_m=0, ~~~~~~~~~\text{for}~~ m\neq -n.$$ 
 
\end{corollary}


\begin{remark} For Corollary \ref{corollary}, note that $\mathcal{T}(R)$ is a holonomic $A_n(K)$-module. So $H^j(\partial, \mathcal{T}(R))$ is a finite dimensional $K$-vector space (see Theorem 6.1 of Chapter 1 in \cite{Bjork}).
\end{remark}

\section{ Koszul homology of generalized Eulerian modules}
The main result of this section is the following:
\begin{theorem}\label{degreezero}
Let $M$ be a generalized Eulerian-$A_n(K)$-module. Then \\ $H_i(X_1, \ldots, X_n; M)$ is concentrated in degree zero.
\end{theorem}
We also give a proof of Theorem \ref{induct}.
\begin{remark}
Note that if $N$ is a holonomic $A_n(K)$-module, then $H_i(\underline{X}; N)$ are finite dimensional vector space. This follows by similar way to Bj${\ddot{\text{o}}}$rk's proof to show that De Rham cohomology of holonomic module is finite dimensional (see Theorem 6.1 of Chapter 1 in \cite{Bjork}).
\end{remark}

Before proving Theorem \ref{degreezero}, we need to prove a few preliminary results.

\begin{proposition}\label{xn}
Let $M$ be a generalized Eulerian $A_n(K)$-module. Then \\ $H_1(X_n; M)$ and $H_0(X_n; M)$ are generalized Eulerian $A_{n-1}(K)$-modules.
\end{proposition}
\begin{proof}
We have an exact sequence of $A_{n-1}(K)$-modules
$$ 0 \longrightarrow H_1(X_n; M) \longrightarrow M(-1) \overset{{X_n}}\longrightarrow M \longrightarrow H_0(X_n; M) \longrightarrow 0.$$
Let $u\in H_1(X_n; M)$ be homogeneous of degree $r$. So we have to prove that
$$ (\mathcal{E}_{n-1} - r)^au = 0\hspace{1cm} \text{for some}~~~ a\geq 1.$$
Notice that $u\in M(-1)_r = M_{r-1}.$ As $M$ is generalized Eulerian, we have

$$ (\mathcal{E}_n- (r-1))^bu=0 \hspace{1cm} \text{for some}~~~ b\geq 1.$$
Since $\partial_n X_n - X_n \partial_n = 1$, we can write

$$\mathcal{E}_n= \mathcal{E}_{n-1} + \partial_n X_n -1.$$
Now we have

$$ (\mathcal{E}_{n-1} - r + \partial_n X_n)^bu=0.$$
Note that $\partial_nX_n$ commutes with $\mathcal{E}_{n-1}.$ Thus
$$0=(\mathcal{E}_{n-1} - r + \partial_n X_n)^bu = (\mathcal{E}_{n-1} - r)^bu + (\ast)\partial_n X_nu = (\mathcal{E}_{n-1} - r)^bu + \alpha X_nu,$$
where $\alpha \in A_n(K)$. Since $X_nu=0$, we get
$$(\mathcal{E}_{n-1} - r)^bu=0.$$
It follows that $H_1(X_n; M)$ is generalized Eulerian $A_{n-1}(K)$-module.

Let  $v\in H_0(X_n; M)$ be homogeneous of degree $s$. Then $v= m + X_nM,$ where $m\in M$ of degree $s$. Because $M$ is generalized Eulerian, we get
$$(\mathcal{E}_n - s)^am =0 \hspace{1cm} \text{for some}~~~ a\geq 1.$$
$$ (\mathcal{E}_{n-1} - s + X_n\partial_n)^am=0.$$
Note that $X_n\partial_n$ commutes with $\mathcal{E}_{n-1}.$ Thus
$$0=(\mathcal{E}_{n-1} - s + X_n\partial_n)^am = (\mathcal{E}_{n-1} - s)^am +  X_n\cdot (\ast)m = (\mathcal{E}_{n-1} - s)^am +  X_n\alpha m,$$
where $\alpha \in A_n(K)$. Going mod  $X_nM$, we get
$$(\mathcal{E}_{n-1} - s)^av=0.$$
It follows that $H_0(X_n; M)$ is generalized Eulerian $A_{n-1}(K)$-module.

\end{proof}

\begin{theorem}\label{x2}
Let $M$ be a generalized Eulerian $A_n(K)$-module. 
Then for $i\geq 2$ and for each  $j\geq 0$, the de Rham homology module $H_j(X_i, X_{i+1}, \dots, X_n; M)$ is a generalized Eulerian $A_{i-1}(K)$-module.
\end{theorem}
\begin{proof}
We prove the result by descending induction on $i$. For $i=n$, the result holds for Proposition \ref{xn}. Set $X=X_i, X_{i+1}, \dots, X_n$ and $X'=X_{i+1}, \dots, X_n$. By induction hypothesis, $N_j=H_j(X'; M)$ are generalized Eulerian $A_i(K)$-module.  By Proposition \ref{xn} again, for $j'=0, 1$ and for each $j \geq 0$, $H_{j'}(X_i; N_j)$ are generalized Eulerian $A_{i-1}(K)$-modules. By Lemma \ref{exact}, we have the exact sequence
$$ 0 \rightarrow H_0(X_i; N_j)\rightarrow H_j(X; M) \rightarrow H_1(X_i; N_{j-1}) \rightarrow 0.$$ 
Since the modules at the left and the right end are generalized Eulerian, by Property \ref{property1} it follows that for each $j\geq0$ the de Rham homology module $H_j(X; M)$ is a generalized Eulerian $A_{i-1}(K)$-module.
\end{proof}
Finally we show:
\begin{proposition}\label{x1}
Let $M$ be a generalized Eulerian $A_1(K)$-module. Then  the modules $H_0(X_1; M)$ and $H_1(X_1; M)$ are concentrated in degree 0.
\end{proposition}
\begin{proof}
We have an exact sequence of $K$-vector spaces
$$ 0 \longrightarrow H_1(X_1; M) \longrightarrow M(-1) \overset{{X_1}}\longrightarrow M \longrightarrow H_0(X_1; M) \longrightarrow 0.$$
Let $\xi \in H_1(X_1; M)(1)$ be homogeneous and nonzero. Since $\xi \in M$ and $M$ is a generalized Eulerian $A_1(K)$-module, we have
$$ (X_1\partial_1 - |\xi|)^a\xi= 0 \hspace{1cm} \text{for some}~~ a\geq 1.$$
Note that $(X_1\partial_1 - |\xi|)^a = (\partial_1X_1 -1 -|\xi|)^a=  (\ast)X_1 + (-1)^a(1+ |\xi|)^a. $ Thus 
$$\alpha X_1\xi + (-1)^a(1+ |\xi|)^a\xi=0,$$ where $\alpha\in A_1(K).$
Since $X_1\xi=0$ and $\xi \neq 0$, we get
$|\xi|=-1$. It follows that $H_1(X_1; M)$ are concentrated in degree 0.

Let $\xi \in H_0(X_1; M)$ be nonzero and homogeneous of degree $r$. Then $\xi = m + X_1M,$ where $m\in M$ of degree $r$. Since $M$ is generalized Eulerian, we get 

$$(X_1\partial_1 - r)^am=0 \hspace{1cm} \text{for some}~~~a\geq 1.$$
$$X_1\cdot (\ast)m + (-1)^ar^am=0. $$

In $M/X_1M$, we have $(-1)^ar^a\xi=0.$ Because $\xi \neq 0,$ we get $r=0. $ It follows that $H_0(X_1; M)$ are concentrated in degree 0.

\end{proof}

Now we give the following:
\begin{proof}[Proof of Theorem \ref{degreezero}]
Set $X'=X_2, \dots, X_n$. By Proposition \ref{x2}, $N_j = H_j(X'; M)$ is generalized Eulerian $A_{1}(K)$-module, for each $j\geq 0$. By Lemma \ref{exact}, we have
$$ 0 \rightarrow H_0(X_1; N_j)\rightarrow H_j(X; M) \rightarrow H_1(X_1; N_{j-1}) \rightarrow 0.$$
for each $j\geq 0.$ By Proposition \ref{x1}, the modules on the left and right of the above exact sequence are concentrated in degree 0. It follows that for each $j\geq 0$ the $K$-vector space $H_j(X; M)$ is also concentrated in degree 0.

\end{proof}
A surprising consequence of Theorem \ref{degreezero} is the following:
\begin{corollary}
Let $M$ be a generalized Eulerian $A_n(K)$-module. If $M$ is also finitely generated as an $R$-module, then $M=0$ or $M=R^m$ for some $m\geq 1$.
\end{corollary}
\begin{proof}
Suppose $M\neq 0$. By Theorem \ref{degreezero}, we have
$$H_0(X_1, \ldots, X_n; M)_j=0   \hspace{1cm} \text{for}~~~ j\neq 0. $$
In particular $M/ (\X) M$  is concentrated in $\deg 0.$
By graded Nakayama's Lemma, we have a surjective minimal map
$$ R^m \overset{{\phi}}\rightarrow M \rightarrow 0.$$
Consider $N=Ker \phi$. If $N\neq 0$, then as $\phi$ is minimal $N_j=0$ for $j \leq 0$.
This forces $H_1(X_1, \ldots, X_n; M)_j\neq 0$ for some $j> 0$. But this is a contradiction. So $N=0$. 
This implies that $M=R^m.$
\end{proof}

Finally we indicate 
\begin{proof}[Proof of Theorem \ref{induct}]
 (1) This follows from Proposition 3.4 in \cite{Puthenpurakal2}. \\
 (2) This follows from Theorem \ref{x2}.\\
 (3) This follows from (1), (2) and Lemma \ref{exact}.
\end{proof}

\section{Proof of Theorem \ref{second}}
In this section we state and prove a more general result which implies Theorem \ref{second}.
\s We  recall a construction from \cite[p.\ 18]{Bjork} which shows that if $M, N$ are $A_n(K)$-modules then 
for all $\nu \geq 0$ the $R$-module $\Tor^R_\nu(M, N)$ has a natural structure of a left $A_n(K)$-module. 
We first note that $M\otimes_R N$ is an $A_n(K)$-module with following action of $\partial_i$:  \[
\partial_i(m\otimes n) = (\partial_i m)\otimes n + m \otimes (\partial_i n).
\]
Now consider  a free resolution  of the left $A_n(K)$-module $M$
\[
\mathbb{F} \colon   \cdots \rt F_n \rt \cdots \rt F_1 \rt F_0 \rt M \rt 0
\]
where $F_i$ are free $A_n(K)$-modules. We note that as $A_n(K)$ is a free $R$-module the complex $\mathbb{F}$ is a resolution of $M$ by free $R$-modules.
It is now clear that $\Tor^R_\nu(M, N) = H_\nu(\mathbb{F}\otimes_R N)$ have natural structure of  $A_n(K)$-modules.
By Theorem 1.6.4 \cite{Bjork}; if $M, N$ are holonomic  $A_n(K)$ modules then so is  $\Tor^R_\nu(M, N)$ for all $\nu \geq 0$

\s If $M, N$ are graded $A_n(K)$-modules, then we can choose $\mathbb{F}$ to be a graded free resolution of $M$ (by graded free $A_n(K)$-modules). So $\Tor^R_\nu(M, N)$ are graded $A_n(K)$-modules.

The main result of this section is
\begin{theorem}\label{tor-R}
Let $M, N$ be graded generalized Eulerian   \ $A_n(K)$-modules. Then $\Tor^R_\nu(M,N)$ is generalized Eulerian for all $\nu \geq 0$.
\end{theorem}
\begin{remark}
The proof of Theorem \ref{tor-R} will show that even if $M, N$ are \textit{Eulerian} $A_n(K)$-modules it does NOT follow that $\Tor^R_\nu(M,N)$  are \textit{Eulerian.}
\end{remark}
We now give
\begin{proof}[Proof of Theorem \ref{tor-R}]
We proceed on the lines of argument given in  the proof of Theorem 1.6.4 given in \cite{Bjork}. 

Let $T = A_n(K) =  K<X_1,\cdots, X_n, \partial_1, \cdots, \partial_n>$ be the original ring we are considering. Let $T^\prime = K<Y_1, \cdots, 
Y_n, \delta_1,\ldots, \delta_n>$ where  $\delta_j = \partial/\partial Y_j$ be another copy of $A_n(K)$. Also
consider the Weyl algebra $S = A_{2n}(K) = K<X_1,\cdots, X_n, Y_1, \cdots, 
Y_n, \partial_1, \cdots, \partial_n, \delta_1,\ldots, \delta_n>$, here $\partial_i = \partial/\partial X_i$ and $\delta_j = \partial/\partial Y_j$.

We consider $N$ as $T^\prime$-module with action of $Y_j$ to be identical as that of $X_j$. 

Now $M \otimes_K N$ can be given a left $S$-module structure as follows
\begin{align*}
X_i(m\otimes n) &= (X_i m)\otimes n,  &Y_j(m\otimes n) &= m\otimes (Y_jn) \\
\partial_i(m\otimes n) &= (\partial_i m)\otimes n,   &\delta_j(m\otimes n) &= m\otimes (\delta_jn).
\end{align*}
Then  Bj$\ddot{o}$rk gives a an isomorphism
\begin{equation*}
\Tor^R_\nu(M, N) \cong H_\nu( Y_1-X_1,\ldots Y_n-X_n, M\otimes_K N  ). \tag{$*$}
\end{equation*}
(Here $M\otimes_K N$ is considered as a $S$-module and $ H_\nu( Y_1-X_1,\ldots Y_n-X_n, M\otimes_K N  ) $
is the $\nu^{th}$-Koszul homology of the $S$-module $M\otimes_K N$ \wrt \  $Y_1-X_1,\ldots, Y_n-X_n$.

\textit{Important Observations:} 
\begin{enumerate}
\item
If $M, N$ are graded $T$-modules then $M\otimes_K N$ is a graded $S$-module.
\item
The isomorphism $(*)$ is graded and of degree zero.
\end{enumerate}
 \textit{Claim:} If $M, N$ are generalized Eulerian $T$-modules then $M\otimes_K N$ is a generalized Eulerian $S$-module.
 
 Note $N$ is a generalized Eulerian $T^\prime$-module.
 Let $u = m\otimes n$ where $m \in M$ and $n \in N$ are homogeneous. Note $|u| = |m| + |n|$. Let $\E_X =  \X^{tr}\partial_X$ be the Eulerian operator \wrt \ $\X$ and let
 $\E_Y$ be the Eulerian operator \wrt \ $\Y$.
 
 As $M$ is a generalized Eulerian $T$-module and $N$ is a generalized Eulerian $T^\prime$-module there exists $a,b \geq 1$ such that
 \[
 (\E_X - |m|)^a m = 0 \quad \text{and} \quad (\E_Y - |n|)^b n = 0. 
 \]
Set $\alpha = \E_X - |m|$ and $\beta = \E_Y - |n|$. Note $\E_X + \E_Y$ is the Eulerian operator on $S$. Further note that $\E_X, \E_Y$ commute with each other. So $\alpha,\beta $ commute with each other.

We now observe that

\begin{align*}
(\alpha + \beta )^{a+b + 1}m\otimes n &= 
\left(\sum_{k = 0}^{a+b+1}\binom{a+b+1}{k} \alpha^{a+b+1 - k}\beta^k \right)(m\otimes n)\\
&= \sum_{k = 0}^{a+b+1}\left(\binom{a+b+1}{k} \alpha^{a+b+1 - k}\beta^k \cdot (m\otimes n) \right) \\
&= \sum_{k = 0}^{a+b+1}\binom{a+b+1}{k}\left( (\alpha^{a+b+1 - k}m) \otimes (\beta^k n) \right) \\
&= 0.
\end{align*}
(The last equality holds since for all $k$ with $0 \leq k \leq a + b + 1$ we have that either $a+b+1 - k \geq a$ OR $ k \geq b$).
It follows that $M\otimes_K N$ is a generalized Eulerian $S$-module.

Consider the change of variables $Z_i = X_i $ for $i = 1,\ldots,n$ and $Z_{n+ j} = X_j - Y_j$ for $j = 1,\ldots,n$. By  \ref{change-equal} and Theorem \ref{induct}(2) it follows that \\ $H_\nu(Z_{n+1}, \ldots, Z_{2n} ; M\otimes_K N)$ is a generalized Eulerian $T$-module. So by $(*)$ it follows that $\Tor^R_\nu(M, N)$ is a generalized Eulerian $T$-module for all $\nu \geq 0$.
\end{proof}
Immediately we get the following result which contains Theorem \ref{second} as a special case. 
\begin{theorem}\label{main-second}
Let $K$ be a field of characteristic zero. Let $R = K[X_1,\ldots, X_n]$ be standard graded.
Let $\mathcal{T}, \mathcal{G}$ be  graded Lyubeznik functor's on $^*Mod(R)$.
Then $\Tor^R_\nu(\mathcal{F}(R), \mathcal{G}(R) )$ is a holonomic  generalized Eulerian $A_n(K)$-module
for all $\nu \geq 0$. 
In particular $H^i_\m(\Tor^R_\nu(\mathcal{F}(R), \mathcal{G}(R) )) \cong E(n)^{a_{i,\nu}} $ 
for some (finite) $a_{i,\nu} \geq 0$.
\end{theorem}
\begin{proof}
By Lyubeznik  \cite[2.9]{Lyubeznik},  we get that $\mathcal{F}(R), \mathcal{G}(R)$ are holonomic $A_n(K)$ modules. So by \cite[1.6.4]{Bjork} we get that $\Tor^R_\nu(\mathcal{F}(R), \mathcal{G}(R) )$ is a holonomic $A_n(K)$-module.

 By \ref{third} we get that $\mathcal{F}(R), \mathcal{G}(R)$ are generalized Eulerian $A_n(K)$-modules. So by Theorem \ref{tor-R} we get that $\Tor^R_\nu(\mathcal{F}(R), \mathcal{G}(R) )$ is a generalized Eulerian $A_n(K)$-module for all $\nu \geq 0$.  
 
 Notice $H^i_\m(-)$ is also a graded Lyubeznik functor. In particular \\ 
 $H^i_\m(\Tor^R_\nu(\mathcal{F}(R), \mathcal{G}(R) )) $ is also a generalized Eulerian 
 $A_n(K)$-module.
  Furthermore it is a holonomic $A_n(K)$-module. It follows that  $H^i_\m(\Tor^R_\nu(\mathcal{F}(R), \mathcal{G}(R) )) \cong E(n)^{a_{i,\nu}} $ for some (finite)  $a_{i,\nu} \geq 0$.
\end{proof}
\section{Proof of Theorem \ref{first}}
In this section we state and prove a more general result which implies Theorem \ref{first}. This result follows easily from Theorem \ref{first-main}.

Let $M$ be a left $A_n(K)$-module. Set $M^\sharp$ to be the standard right $A_n(K)$-module associated to $M$, see \ref{left-right}.
The main result of this section is
\begin{theorem}\label{first-main}
Fix $\nu \geq 0$.
Let $M,N$ be  holonomic generalized Eulerian left $A_n(K)$-modules. Then 
$\Tor^{A_n(K)}_\nu(M^\sharp, N)$ is concentrated  in degree $-n$, i.e., \\
$\Tor^{A_n(K)}_\nu(M^\sharp, N)_j = 0 $ for $j \neq -n$.
\end{theorem}
\begin{proof}
We have to first carefully de-construct the proof by Bj$\ddot{o}$rk, (see proof of Theorem 1.6.6, \cite{Bjork}) showing that if $U$ is a holonomic  right $A
_n(K)$-module and $V$ is a holonomic  left
 $A_n(K)$-module  then $\Tor^{A_n(K)}(U, V)$ is a finite dimensional $K$-vector space for all $\nu \geq 0$.
 
 As in Theorem \ref{main-second} we consider the following Weyl-algebra's: \\
 (a) $T = A_n(K) =  K<X_1,\cdots, X_n, \partial_1, \cdots, \partial_n>$ be the original ring we are considering. \\
 (b) Let $T^\prime = K<Y_1, \cdots, 
Y_n, \delta_1,\ldots, \delta_n>$ where  $\delta_j = \partial/\partial Y_j$ be another copy of $A_n(K)$.  \\
(c) Also
consider the Weyl algebra 
$$S = A_{2n}(K) = K<X_1,\cdots, X_n, Y_1, \cdots, Y_n, \partial_1, \cdots, \partial_n, \delta_1,\ldots, \delta_n>,$$
 here $\partial_i = \partial/\partial X_i$ and $\delta_j = \partial/\partial Y_j$.

We consider $V$ as $T^\prime$-modules with action of $Y_j$ to be identical as that of $X_j$. 

Now $U \otimes_K V$ can be given a \textit{left} $S$-module structure as follows
\begin{enumerate}
\item
$X_i(u\otimes v) = (X_i u)\otimes v$,
\item
$Y_j(u\otimes v) = u\otimes (vn)$,
\item
$\partial_i(u\otimes v) = - ( u \partial_i)\otimes v$,
\item
$\delta_j(u\otimes v) = m\otimes (\delta_jn)$.
\end{enumerate}
The minus sign in (3) is essential if we have to give left $S$-module structure on $U\otimes_K V$. In  proof of  Theorem \cite[1.6.4]{Bjork} it is shown that $U\otimes_K V$ is a holonomic $S$-module.

Now if we give $M^\sharp\otimes_K N$ the left $S$-module structure as above then only action (3) deserves a comment. Notice
\begin{align*}
\partial_i(m\otimes n) &= - ( m \partial_i)\otimes n,\\
&= -(\tau(\partial_i) m)\otimes n, \\
&= -(-\partial_i m)\otimes n, \\
&= (\partial_i m)\otimes n. \tag{$\dagger$}
\end{align*}

\textit{Claim:} If $M, N$ are generalized Eulerian $T$-modules then $M^\sharp\otimes_K N$ is a generalized Eulerian $S$-module.
 
 Note $N$ is a generalized Eulerian $T^\prime$-module.

As $M$ is a generalized Eulerian $T$-module and $N$ is generalized Eulerian $T^\prime$-module there exists $a,b \geq 1$ such that
 \[
 (\E_X - |m|)^a m = 0 \quad \text{and} \quad (\E_Y - |n|)^b n = 0. 
 \]
Set $\alpha = \E_X - |m|$ and $\beta = \E_Y - |n|$. Note $\E_X + \E_Y$ is the Eulerian operator on $S$. Further note that $\E_X, \E_Y$ commute with each other. So $\alpha,\beta $ commute with each other.

We now observe that

\begin{align*}
(\alpha + \beta )^{a+b + 1}&m\otimes n = 
\left(\sum_{k = 0}^{a+b+1}\binom{a+b+1}{k} \alpha^{a+b+1 - k}\beta^k \right)(m\otimes n)\\
&= \sum_{k = 0}^{a+b+1}\left(\binom{a+b+1}{k} \alpha^{a+b+1 - k}\beta^k \cdot (m\otimes n) \right)   \\
&= \sum_{k = 0}^{a+b+1}\binom{a+b+1}{k}\left( (\alpha^{a+b+1 - k}m) \otimes (\beta^k n) \right) \   \ \text{(use equation $(\dagger)$)},  \\
&= 0.
\end{align*}
(The last equality holds since for all $k$ with $0 \leq k \leq a + b + 1$ we have that either $a+b+1 - k \geq a$ OR $ k \geq b$).
It follows that $M^\sharp\otimes_K N$ is a generalized Eulerian $S$-module.

We note the elements $X_1 - Y_1,\cdots, X_n - Y_n, \partial_1 + \delta_1, \cdots,
\partial_n + \delta_n$ commute with each other. By proof of \cite[Theorem 1.6.6]{Bjork}  for all $\nu \geq 0$ we have a degree zero isomorphism of vector spaces
\[
\Tor^{A_n(K)}_\nu(M^\sharp, N) \cong H_\nu(X_1 - Y_1,\cdots, X_n - Y_n, \partial_1 + \delta_1, \cdots,\partial_n + \delta_n; M^\sharp\otimes_K N).
\]
(the latter is Koszul homology of $M^\sharp\otimes_K N$  \wrt \ commuting operators  $X_1 - Y_1,\cdots, X_n - Y_n, \partial_1 + \delta_1, \cdots,
\partial_n + \delta_n$).

Consider the change of variables
\[
Z_i = X_i + Y_i \quad W_i = X_i - Y_i \quad \text{for} \ i = 1,\ldots, n.
\] 
We note that 
\[
\frac{\partial}{\partial Z_i} = 
\frac{1}{2} \left( \frac{\partial}{\partial X_i} + \frac{\partial}{\partial Y_i}  \right).
\]
\s Thus for all $\nu \geq 0$ we have an isomorphism of graded $K$-vector spaces,
\begin{equation*}
\begin{split}
H_\nu(X_1 - Y_1,\cdots, &X_n - Y_n, \partial_1 + \delta_1, \cdots,\partial_n + \delta_n; M^\sharp\otimes_K N)   \\
\cong H_\nu(W_1, \ldots, &W_n, \partial/\partial Z_1, \ldots, \partial/\partial Z_n \colon M^\sharp\otimes_K N). 
\end{split} 
\end{equation*}
The result now follows from Theorem \ref{induct}(3).
\end{proof}

Immediately we get the following result which contains Theorem \ref{first} as a special case. 
\begin{theorem}\label{main-first}
Let $K$ be a field of characteristic zero. Let $R = K[X_1,\ldots, X_n]$ be standard graded.
Let $\mathcal{T}, \mathcal{G}$ be  graded Lyubeznik functors on $^*Mod(R)$. Fix $\nu \geq 0$.
Then   $\Tor^{A_n(K)}_\nu(\mathcal{F}(R)^\sharp, \mathcal{G}(R) )$ is a
 finite dimensional, graded $K$-vector-space
 concentrated
in degree $-n$
\end{theorem}
\begin{proof}
By Lyubeznik  \cite[2.9]{Lyubeznik},  we get that $\mathcal{F}(R), \mathcal{G}(R)$ are holonomic $A_n(K)$ modules.  So by \cite[1.6.6]{Bjork} we get that 
$\Tor^{A_n(K)}_\nu(\mathcal{F}(R), \mathcal{G}(R) )$ is a  finite dimensional $K$-vector space.

 By \ref{third} we get that $\mathcal{F}(R), \mathcal{G}(R)$ are generalized Eulerian $A_n(K)$-modules. The result now follows from Theorem \ref{first-main}.
\end{proof}

\section{Conjecture \ref{conj}}
In this section we give evidence which supports our conjecture \ref{conj}.

\s \label{bjork-ext} \emph{We first recall the proof by Bj$\ddot{o}$rk} which shows that if  $M, N$ are holonomic $A_n(K)$-modules then $\Ext^\nu_{A_n(K)}(M, N)$ is a finite dimensional $K$-vector space for all $\nu \geq 0$.

$\bullet$ \emph{Step-1:} Consider the dual of $M$ namely $M^\dagger = \Ext^n_{A_n(K)}(M, A_n(K))$. Then $M^\dagger$ is a right holonomic $A_n(K)$-module, see 7.12, Chapter 2, \cite{Bjork}.

$\bullet$ \emph{Step-2:} For all $\nu \geq 0$ there is an isomorphism  
(see 7.15, Chapter 2, \cite{Bjork}):
$$\Ext^\nu_{A_n(K)}(M, N) \cong \Tor^{A_n(K)}_{n-\nu}(M^\dagger, N). $$
It now follows from Theorem 6.6, Chapter 1, \cite{Bjork};  that $\Ext^\nu_{A_n(K)}(M, N)$ is a finite dimensional $K$-vector space for all $\nu \geq 0$.

\begin{remark}
If $M, N$ are graded holonomic $A_n(K)$-modules then the isomorphism in Step 2 is homogeneous of degree zero. 
This can be seen by inspecting the proof; (also note the isomorphism given in Proposition 4.14, Chapter 2, \cite{Bjork} is also homogeneous of degree zero). 
\end{remark}

We first indicate another conjecture which solves our conjecture \ref{conj}.
\begin{conjecture}\label{dual}
Let $M$ be a non-zero, left, holonomic, graded, generalized Eulerian  $A_n(K)$-module. Then $^\sharp M^\dagger (+n)$ is a left generalized Eulerian $A_n(K)$-module.
\end{conjecture}
\begin{remark}
We recall that if $N$ is a right $A_n(K)$-module then $^\sharp N$ is the standard
left module associated to $N$.
\end{remark}

\s \label{conj2-1} We now show
\begin{proof}[ How Truth of Conjecture \ref{dual} implies  validity of Conjecture \ref{conj}]
 Set $L = \  ^\sharp M^\dagger$. \\ Then $L^\sharp = M^\dagger$. Also note that $L(+n)^\sharp \cong L^\sharp(+n)$.
 
 If $L(+n)$ is generalized Eulerian then by Theorem \ref{first-main} we get that for all $\nu \geq 0$ the finite dimensional 
 $K$-vector space
 $\Tor^{A_n(K)}_{n - \nu}(L^\sharp(+n), N)$ is concentrated in degree $-n$. 
 It follows that $\Tor^{A_n(K)}_{n - \nu}(L^\sharp, N)$ is concentrated in degree $0$. \\ So
$\Tor^{A_n(K)}_{n - \nu}(M^\dagger, N)$ is concentrated in degree $0$.

By \ref{bjork-ext}(Step 2) we get the result.
 \end{proof}

 \s \label{con}\textit{Convention}. Set $D = A_n(K)$. Let $ ^lD$ (and $D^r$) be $D$ considered as a left $D$-module (right $D$-module).
 Similarly let $ ^lR$ (and $R^r$) be $R$ considered as a left $D$-module (right $D$-module).
 
 \s \label{facts} Let us recall the following well-known facts. \\
 (i) Let $M$ be a left $D$-module. Then $\Hom_D(M,D)$ is a right $D$-module with $D$-action given as follows:
 Let $f \in \Hom_D(M, D)$ and $d \in D$. Define $fd \colon M \rt D$ 
 \[
  (fd)(m) = f(m)d \quad \text{for all} \ m \in M.
 \]
(ii) The map $ \psi \colon \Hom_D( ^lD, D) \rt D^r$ defined by $\psi(f) = f(1)$ is an isomorphism (as right $D$-modules).
 \begin{lemma}\label{dual-R}[with notation as in \ref{bjork-ext} and \ref{con}]
  \[
   (^lR)^\dagger \cong R^r(-n).
  \]
 \end{lemma}
\begin{proof}
 It is well known that the Koszul complex ($\K$) of $ ^lD$ \wrt \ $\partial_1, \ldots, \partial_n $ is acyclic.
 Note $H_0(\K) = \  ^lR$. So $\K$ is a free resolution of $ ^lR$ (as left $D$-modules). We note that
 $\K_n = \ ^lD(+n)$ and $\K_{n-1} = \ ^lD(+(n-1))^n$.
 
 We consider the complex $\K^* = \Hom_D(\K, D)$.  Ignoring shifts  it is elementary to see that $H^n(\K^*) = R^r$.
 As shifts are vital to us, we note that $\K^*_n = D^r(-n)$. Thus $H^n(\K^*) = R^r(-n)$. 
 
 Finally note that $\Ext^n_D( ^lR, D) = H^n(\K^*)  = R^r(-n)$.
\end{proof}

\begin{example}\label{ext-de-rham}
 Let $\mathcal{T}$ be a graded Lyubeznik functor on $^*Mod(R)$. 
Then de Rham cohomology  module $H^j(\partial, \mathcal{T}(R))$ is concentrated in degree $-n$ (see \ref{corollary}).
We note that we have an isomorphism of graded $K$-vector spaces
\begin{equation*}
H^\nu\left(\partial, \T(R) \right)  \cong \Tor^{D}_{n-\nu}(R^r, \T(R)).  \tag{ i }
\end{equation*}
It follows that 
$\Tor^{D}_{n-\nu}(R^r, \T(R))(-n)$  is concentrated in degree zero.
By \ref{bjork-ext}(2) we have a graded isomorphism
$$\Ext^\nu_{D}( \ ^lR, \T(R)) \cong \Tor^{D}_{n-\nu}((\ ^lR)^\dagger, \T(R)). $$
By Lemma \ref{dual-R} we get that $(\ ^lR)^\dagger \cong  R^r(-n)$.
Thus we have a graded  isomorphism
\[
 \Ext^\nu_{D}( \ ^lR, \T(R)) \cong  H^\nu\left(\partial, \T(R) \right)(-n)
\]
Thus $ \Ext^\nu_{D}( \ ^lR, \T(R))$ is concentrated in degree zero.
\end{example}

\s \textit{Evidence for validity of Conjecture \ref{conj}:}
Fix $\nu \geq 0$. Suppose there exists $m \in \mathbb{Z}$ such that $\Ext^\nu_D(M, N)$ is concentrated in degree $m$ for all
 left, graded, holonomic, generalized Eulerian $D$-modules   $M$ and  $N$. We show that necessarily $m = 0$.
 
 \begin{enumerate}
  \item $\nu = 0:$ Let $M$ be a a \textit{non-zero} left, graded, holonomic, generalized Eulerian $D$-module. Then the identity
  map $1_M \in \Hom_D(M,M)$ and is \emph{non-zero}.
  \item $\nu = 1:$  Let $f$ be an irreducible polynomial in $R$. Let $P = (f)$.  Consider the exact sequence of $R$-modules:
  \[
   0 \rt \  ^lR \rt  \ ^lR_f \rt H^1_{P}(R) \rt 0.  \tag{*}
  \]
It is easily verified that this is in fact an exact sequence of left $D$-modules. If (*) is split  sequence of $D$-modules then
it is split as $R$-modules. In particular we get that $P$ is an associate prime of $R_f$, a contradiction.
Set $M = H^1_P(R)$ and $N = \ ^lR$. Then the exact sequence  (*) is a non-zero element in
$\Ext^1_D(M,N)_0$.
\item $\nu = n:$ Let $N_n = H^n_\m(R)$. Then by \cite[Theorem 2.4]{Puthenpurakal1} it follows that
$H^n(\partial, N_n) = K$ (in degree $-n$). So we have that $\Ext^n_D(\ ^lR, N_n)$ is non-zero and concentrated in degree
$0$.
\item $\nu = n-1:$  Set $P_{n-1} = (X_1,\ldots, X_{n-1})$.  Set $N_{n-1} = H^{n-1}_{P_{n-1}}(R)$. Then by 
\cite[Theorem 4.3]{Puthenpurakal1} it follows that
$H^{n-1}(\partial, N_{n-1}) = K$ (in degree $-n$). So we have that $\Ext^{n-1}_D(\ ^lR, N_{n-1})$ is non-zero and concentrated in degree
$0$.
\item $1 \leq \nu \leq n-2$. Set $P_\nu = (X_1,\ldots, X_{\nu})$.  Set $N_{\nu} = H^{\nu}_{P_{\nu}}(R)$. Then by 
techniques similar to  \cite[Theorem 2.4]{Puthenpurakal1}, 
\cite[Theorem 4.3]{Puthenpurakal1}  (and a little tedious induction) it follows that
$H^{\nu}(\partial, N_{\nu}) = K$ (in degree $-n$). So we have that $\Ext^{\nu}_D(\ ^lR, N_{\nu})$ is non-zero and concentrated in degree
$0$.
\item $\nu = 2$. Let $I$ be a graded ideal of height $g$ such that $H^i_I(R) = 0$ for $i > g + 1$ and $H^{g+1}_I(R) = E(n)^a$ for some
$a \geq 1$ (for a specific example see example 6.1 in \cite{W}). 
Note $H^g_I(R)$ is non-zero and say has associate primes $P_1, \ldots,P_c$. Then note that $P_i$ are the minimal primes of $I$
of height $g$ (in particular they are graded and not equal to $\m$). Let $f \in \m \setminus \cup_{i = 1}^{c}P_i$
be homogeneous. We have an exact
sequence of $R$-modules (and also of $D$-modules)
\[
 0 \rt H^g_I(R) \rt H^{g}_I(R)_f \rt H^{g+1}_{(I,f)} \rt E(n)^a \rt E(n)^a_f = 0. \tag{ **}
\]
Clearly we have that $(**)$ is an element of $\Ext^2_D(E(n)^a, H^g_I(R))_0$. I believe it is non-zero. Unfortunately
I do not know how to show it.
 \end{enumerate}

\bibliographystyle{amsplain}

\end{document}